\documentclass[11pt,a4paper]{amsart}
\usepackage{commands_Ryo, packages_Ryo,packages,commands}

\title{Decomposition theorem for perfectoid rings along general ideals}

\author[R. Ishizuka]{Ryo Ishizuka}
\address{Department of Mathematics, Institute of Science Tokyo, 2-12-1 Ookayama, Meguro, Tokyo 152-8551}
\email{ishizuka.r.ac@m.titech.ac.jp}

\author[L. Navarro Chafloque]{Léo Navarro Chafloque}
\address{EPFL SB MATH CAG,
MA C3 585 (Bâtiment MA),
Station 8,
1015 Lausanne}
\email{leo.navarrochafloque@epfl.ch}



\begin{document}

\begin{abstract}

Using André's lemma and the excision square for perfectoidization coming from \(p\)-complete arc descent, we prove new structural results about perfectoid rings and perfectoidization. 

The main result is a tameness theorem for torsion in perfectoid rings: if \(R\) is a perfectoid ring and \(I\subset R\) is an ideal, then the \(I\)-torsion in \(R\) is \(I_{\pfd}\)-almost zero. This yields an excision-type decomposition of \(R\) along its \(I\)-torsion part. We also study (semi)perfectoid rings and perfectoid ideals and take the opportunity to make some structural remarks about them.
\end{abstract}

\maketitle

\tableofcontents

\section{Introduction}

Perfectoid rings were introduced by Scholze in \cite{scholze2012Perfectoida} and have since played a central role in mixed characteristic algebraic geometry. Their applications include the development of prismatic cohomology \cite{bhatt2018Integral,bhatt2022Prismsa}, the theory of the Fargues--Fontaine curve and its role in the geometrization of the Local Langlands correspondence \cite{fargues2024geometrizationlocallanglandscorrespondence}, as well as breakthroughs in commutative algebra \cite{andre2018Conjecture,bhatt2021cohenmacaulaynessabsoluteintegralclosures} and birational geometry \cite{bhatt2022globallyregularvarietiesminimal,takamatsu2023Minimal}.

This paper studies perfectoidization as a structural operation on rings and we prove structural results on torsion phenomena in perfectoid rings. Our main tools are André's lemma, originating in \cite{andre2018Conjecture} and appearing in the form \cite[Theorem 7.14]{bhatt2022Prismsa}, together with the \(p\)-complete arc descent property of perfectoidization \cite[Section 8.2]{bhatt2022Prismsa}. We use these tools to revisit several known results and to prove new structural statements about perfectoid rings.

The first half of this paper (\Cref{arc_recollection} and \Cref{semipfd_arcproof}) treats some fundamental properties of perfectoidization focusing on arc-topological perspective.
In particular, we will show in \Cref{semipfdproot} that the perfectoidization of a $p$-torsion-free semiperfectoid ring is given by the $p$-completion of its $p$-root closure in the $p$-localization, which was previously established in \cite[Theorem 1.3]{ishizuka2024Calculation} but the proof is much simpler.

In \Cref{torsion_section}, we establish the main results of this paper. After some technical prerequisites about $p$-completely faithfully flat base change (\Cref{tech_pcom_subsec}) we establish in \Cref{tame_torsion_section} that torsion in perfectoid rings satisfies a strong tameness property -- this saying taking its precise form in the theorem below.
\begin{thm}[{\Cref{torsion_is_almostzero}}]  \label{thm_torsion_intro}
    Let $R$ be a perfectoid ring and let \(I\) be an ideal of \(R\). Then the $I$-torsion submodule \(R[I] \defeq \set{x \in R}{Ix = 0}\) in $R$ is $I_{\pfd}$-almost zero.
\end{thm}

As a consequence of \Cref{thm_torsion_intro} the following decomposition theorem along any ideal $I\subset R$ of a perfectoid ring.

\begin{thm}[{\Cref{dtorsion_decomposition}}]  \label{thm_decomposition_intro}
The natural map
\[R\to R/R[I]\times_{R/(R[I],I_{\pfd})}R/I_{\pfd}\]
is an isomorphism and provides a decomposition of $R$ into a $I$-torsion-free perfectoid ring and a $I$-torsion perfectoid ring.
\end{thm}

In particular, the above includes the fact that $R[I]$ is a perfectoid ideal, from which we establish other properties in \Cref{bounded_d_infinity_torsion}.

In \Cref{exist_pfd_sec}, we apply the above results to the existence problem for discrete perfectoidizations. We show in \Cref{I_tors and I_tors_free existence} that this question can be studied by excision along an arbitrary ideal \(I\subset R\).

\subsubsection*{Structure of this paper}
In \Cref{arc_recollection}, we recall how the arc topology interacts with perfectoid rings.



In \Cref{semipfd_arcproof}, we prove \Cref{semipfdproot} -- in this section, we also take the opportunity to prove some structural results on semiperfectoid rings.

In \Cref{pfdideals_section}, we will study ideals of perfectoid rings whose quotients become again perfectoid rings. This will collect and improve previous results on such ideals.


In \Cref{torsion_section}, we will prove our main results and their consequences stated above.

\subsubsection*{Notations and setup} Throughout, $p\in \N$ is a fixed prime number. 
If $R$ is a ring, and $I\subset R$ is an ideal and $M$ an $R$-module, we define submodules
\begin{align*}
    M[I] & \defeq \set{m \in M}{\text{\(xm = 0\) for all \(x \in I\)}} \quad \text{and} \\
    M[I^{\infty}] & \defeq \set{m \in M}{\text{There exists \(N > 0\) such that \(xm = 0\) for any \(x \in I^N\)}}.
\end{align*}
If \(R\) is a perfectoid ring, we can define the perfectoidization \(I_{\pfd}\) of ideal \(I\) as the kernel of the surjective map \(R \to (R/I)_{\pfd}\) (\Cref{def_pfdization_ideal}).
We say that $M$ is \emph{$I_{\pfd}$-almost zero} if \(M\) is equal to \(M[I_{\pfd}]\) (similarly as \cite[Definition 10.1]{bhatt2022Prismsa}).

In \Cref{tech_pcom_subsec}, we use the following conventions. If $R$ is a derived $p$-complete ring then we write $\widehat{\Ds}(R)$ for the \emph{stable infinity category of derived $p$-complete $R$-modules.} If $R\to R'$ is a ring morphism between derived $p$-complete rings, we write by
$R'\widehat{\otimes}^L_R(-)$ for the derived $p$-complete base change which is a colimit preserving functor $\widehat{\Ds}(R)\to \widehat{\Ds}(R')$.


\subsection*{Acknowledgments}
This work was started in the problem session of the conference `\(p\)-adic and Characteristic \(p\) Methods in Algebraic Geometry' at EPFL. We are very grateful for these opportunities and their hospitality. The second-named author thanks Kęstutis Česnavičius for interesting discussions at this same conference. The authors thank Marta Benozzo for intersecting discussions at this conference on the subject of the paper. The second-named author thanks Lucas Gerth, Alapan Mukhopadhyay and Zsolt Patakfalvi for helpful and interesting discussions on the subject of the article. The authors thanks Dimitri Dine for interesting feedback on an earlier draft of this work.
The first-named author was supported by JSPS KAKENHI Grant number 24KJ1085. The second-named author was supported by the FNS grant, \#200021-
231484.


\section{Recollection on arc covers and perfectoidization}\label{arc_recollection}

We recall how a permissive notion of covers can be used to understand perfectoidization. 

\begin{rem}\label{arc_top_def} We consider the \emph{$p$-complete arc topology} on commutative rings, see \cite[Section 2.1.1]{cesnavicius2023purityflatcohomology}. See in particular the second paragraph there which implies that in definition the of $p$-complete arc-covers, we may test that a map is an arc cover only with perfectoid valuation rings of dimension $\leq 1$. Note also that the $p$-complete arc site of a ring $R$ is the same as the $p$-complete arc site of the classical $p$-completion $R^{\wedge p}$ of $R$.
    
\end{rem}

\begin{rem}\label{pfdization_def} If $R$ is any ring, we define the \emph{perfectoidization} of $R$ as 
\[R_{\pfd}=R\Gamma_{\arc}(\Spf(R^{\wedge p}),\Os)\]
which is in general a coconnective $E_{\infty}$-ring. We will be interested in this article in cases where it is discrete, see \cite[Corollary 7.3 and Theorem 10.11]{bhatt2022Prismsa}. If $R_{\pfd}$ is discrete, and $R^{\wedge p}$ admits a map from a perfectoid ring, then the map $R\to R_{\pfd}$ is initial with respect to maps $R\to S$ where $S$ is a perfectoid ring, see \cite[Corollary 8.14]{bhatt2022Prismsa}. Note also that we have that by \cite[Corollary 8.11]{bhatt2022Prismsa}
\[R_{\pfd}=\varprojlim_{\substack{R\to R' \\ R' \pfd}}R'\]
in the $\infty$-category of $E_{\infty}$-rings. So if this limit happens to be a discrete perfectoid ring, it follows that $R\to R_{\pfd}$ is initial with respect to maps to perfectoid rings. Reciprocally, if $R$ is a ring such that there exist an inital perfectoid ring $R\to S$ over $R$, then the limit expression forces $R_{\pfd}=S$.
\end{rem}

The following lemma is useful to realize perfectoidization as subrings of big perfectoid rings.

\begin{lemma}\label{archull} Let $R$ be any $p$-complete ring that admits a map $R\to R_{\pfd}$, where $R_{\pfd}$ is a discrete perfectoid ring, this map being initial with respect to maps to perfectoid rings. Say $R\to S$ is a $p$-complete arc cover where $S$ is a perfectoid ring. Then the natural map $R_{\pfd}\to S$ is injective.

\end{lemma}
\begin{proof}
We essentially only need to explain why the natural map $R_{\pfd}\to S$ is also a $p$-complete arc-cover. Let $R_{\pfd}\to V$ where $V$ is a perfectoid valuation ring. Then by hypothesis, considering the precomposition $R\to V$, there is an extension of perfectoid valuation rings $V\to V'$ and a commutative square
\[\begin{tikzcd}
R \arrow[d] \arrow[r] & S \arrow[d] \\
V \arrow[r]           & V'         
\end{tikzcd}\]
Because every ring except $R$ is perfectoid in this diagram, a similar diagram with $R_{\pfd}$ replaced by $R$ and the induced maps also exists using our hypothesis that $R\to R_{\pfd}$ is initial with respect to maps to perfectoid rings, which shows the claim.
    Then $R_{\pfd}\to S$ is injective a consequence of \cite[ Proposition 8.10]{bhatt2022Prismsa}.
\end{proof}

\begin{rem}\label{arc_remarks}We point out interesting types of $p$-complete arc covers of any ring $R$.
\begin{enumerate}
    \item The first one is from the proof \cite[Lemma 8.8]{bhatt2022Prismsa}. Namely, it is rather easy to take a $p$-completely faithfully flat cover by a semiperfectoid ring $R\to R'$ -- namely take the $p$-completion of a faithfully flat cover $R\to R'$ where $R'/pR'$ is semiperfect and $p$ admits a $p$-th root in $R'$. The argument in \Cref{archull} shows that $R'\to R'_{\pfd}$ is an arc cover. Therefore $R\to R'_{\pfd}$ is an arc-cover. 
    \item Following, \cite[Remark 8.9]{bhatt2022Prismsa} we can provide a canonical (up to choice of algebraic closures) $p$-complete arc cover of any ring. Write $\Spa(R,R)_{\leq 1}$ for the set of continuous valuations of rank $\leq 1$ on $R$ equipped with the $p$-adic topology. This set is identified with the subset of $\Spec(R)$ which consists of primes that are supports of (rank 1) perfectoid valuations. For $x\in \Spa(R,R)_{\leq 1}$ write $V_{x}$ for the $p$-completion of an absolute integral closure of the corresponding valuation ring to $x$ in $k(\supp(x))$. Then $R\to \prod_{x\in \Spa(R,R)_{\leq 1}}V_x$ is a $p$-complete arc cover to a perfectoid ring. 
\end{enumerate}

Using the above, we can revisit  \cite[Lemma A.2]{ma2022Analogue} -- their proof does not use the arc topology as a core engine, where ours does. Note that technical argument in both proofs is similar.

\begin{cor}\label{pfd_is_p_torsionfree} Suppose that $R$ is a $p$-torsion-free ring which admits a discrete perfectoidization. In the construction of \Cref{arc_remarks}, (2), on one can take only those $V_x$'s that are $p$-torsion-free in the construction exposed in (2) above, and the map stays injective. Namely we have an injective map
\[R_{\pfd}\to \prod_{\substack{x\in \Spa(R,R)_{\leq 1} \\ p\not \in \supp(x) }}V_x.\]
In particular $R_{\pfd}$ is $p$-torsion-free.
\end{cor}
\begin{proof}
    Take a valuation $y\in \Spa(R,R)_{\leq 1}$ with $p\in \supp(y)$. Because $R$ is $p$-torsion-free, $p$ is not contained in any minimal prime of $R$. So there is a prime $\qk\subset \supp(y)$ with $p\not \in \qk$. Take a valuation ring $V$ with support $\qk$ dominating the prime $\supp(y)$. Take a valuation ring extension of $V$ of rank 1, that we denote by $W$. This valuation ring $W$ is this is still $p$-torsion-free -- indeed it is a faithfully flat extension of a $p$-torsion-free domain. Then take the $p$-adic completion $\widehat{W}$ which is again still a $p$-torsion-free domain because the $p$-completion of any $p$-torsion-free module is $p$-torsion-free. Now $\widehat{W}$ defines a rank 1 valuation $z$ on $R$ with $p\not \in \supp(z)\subset \supp(y)$. Therefore we can harmlessly remove $y$ from the set where we take the product and the map $R_{\pfd}\to \prod V_x$ stays injective. 
\end{proof}
    
\end{rem}
\begin{rem}
     If $R$ is a a perfectoid ring and $I$ an ideal of $R$ then \Cref{archull} shows that (see \Cref{def_pfdization_ideal})
    \[I_{\pfd}=\bigcap_{x\in \Spa(R/I,R/I)_{\leq 1}}\supp(x)\]
    which is to compare to \cite[Proposition 5.6]{Dine_2024} which states an analogous result for $p$-torsion-free perfectoid rings.
\end{rem}

The following lemma explains why we can use excision in the arc topology \cite[Corollary 8.12]{bhatt2022Prismsa}, in the case of integral maps, which can be useful to compute perfectoidizations. We thank Kęstutis Česnavičius for highlighting the usefulness of this lemma to compute the perfectoidization of $\mu_p$. The authors also use this result in \cite{pfd_fin_alg}.


\begin{proposition}[{\cite[Corollary 8.12]{bhatt2022Prismsa}}]{\label{812prisms}}Let $R$ be a perfectoid ring. Suppose that $R\to S$ is the $p$-completion of an integral ring morphism.\footnote{This assumption ensures that all perfectoidization appearing in the pullback square of Lemma \ref{812prisms} are discrete.} Let $S\to S'$ be also the $p$-completion of an integral map $S_0\to S_0'$. Let $I\subset S_0$ an ideal 
such that $S_0\to S_0'$
is an isomorphism away from $I$. Then we have (derived) pullback square of rings
\[\begin{tikzcd}
S_{\pfd} \arrow[d] \arrow[r] & S'_{\pfd} \arrow[d] \\
(S/I)_{\pfd} \arrow[r]       & (S'/I)_{\pfd}      
\end{tikzcd}\]
    
\end{proposition}
\begin{proof}
First, we note that in \cite[Corollary 8.12]{bhatt2022Prismsa} an assumption about derived $p$-completeness of $I$ is made. This is done there because the authors only consider the $p$-complete arc site on derived $p$-complete rings. Note that our definition of perfectoidization of a ring $R$ is anyway the perfectoidization of the classical completion $R^{\wedge p}$, see \Cref{pfdization_def} so it aligns with their setup. This is why we can disregard the hypothesis on $I$. Indeed, anyway, as explained below, using our hypothesis
\[S_0\to S_0'\times S_0/I\]
is an arc-cover. Therefore, $p$-completing the above, we get that
\[S\to S'\times S/\overline{IS}\]
is a $p$-complete arc-cover, where $\overline{IS}$ denotes the $p$-adic closure of $IS$ in $S$ -- this ideal is derived $p$-complete. This last setup is the only thing which is really needed for the proof of \cite[Corollary 8.12]{bhatt2022Prismsa}. Note also that using our definition 
\begin{equation*}
    (S_0)_{\pfd}=(S)_{\pfd} \quad  (S_0')_{\pfd}=(S')_{\pfd} \quad (S_0/I)_{\pfd}=(S/\overline{IS})_{\pfd} \quad (S_0'/IS_0')_{\pfd}=(S'/\overline{IS'})_{\pfd}.
\end{equation*}

We therefore only need to recall why using our hypothesis, $S_0\to S_0'$ is an isomorphism away from $I$ in the arc topology. So let $V$ be a valuation ring of rank 1 and $S_0\to V$ a map -- we have to show that if $I$ is not sent to zero in $V$, then there is a unique extension
\[\begin{tikzcd}
S_0 \arrow[d] \arrow[r] & S_0' \arrow[ld, dashed] \\
V                     &                      
\end{tikzcd}\]
Let $i\in I$ such that the image of $i$ is not zero in $V$. By assumption 
\[S_0\left[i^{-1}\right]\to S_0'\left[i^{-1}\right]\]
is an isomorphism. Now, it suffices to show that the map
\[S_0'\to S_0'\left[i^{-1}\right]\xleftarrow{\sim}S_0\left[i^{-1}\right]\to V\left[i^{-1}\right]\subset \Frac(V),\]
factors through $V$. 
Take $s\in S_0'$. Then $s$ is integral over $S_0$. Therefore it follows that the image of $s\in \Frac(V)$ is integral over $V$. But as $V$ is integrally closed in $\Frac(V)$, we get the result.


    
\end{proof}

\begin{rem}To apply \Cref{812prisms}, one can sometimes take $S_0=S\to S_0'=S'$, so the case where we have an integral map between already $p$-complete rings. Both setups can be useful.
\end{rem}

\section{Semiperfectoid rings}\label{semipfd_arcproof}
The goal of this section is to prove \Cref{semipfdproot}. We take the opportunity to make some remarks about semiperfectoid rings beforehand.

\subsection{Intrinsic definition of $p$-complete semiperfectoid rings}

Recall the following from \cite[Section 7]{bhatt2022Prismsa}.

\begin{definition}\label{def_semipfd} A ring $S$ is called \emph{semiperfectoid} if it is derived $p$-complete and that there exists a surjection $R\to S$ from a perfectoid ring.
\end{definition}

In \cite[Definition 3.5.7 and Lemma 3.5.8]{bhatt2025Aspects}, Bhatt gives another definition of semiperfectoid rings and proves that it can be written as a quotient of perfectoid rings.
Below (\Cref{intrisic_semipfd}), we show that there is an intrinsic definition of $p$-complete semiperfectoid rings that does not refer to a (surjective) map from a perfectoid ring, and that the definition \cite[Lemma 3.5.8]{bhatt2025Aspects} is equivalent to \Cref{def_semipfd} in the classical $p$-complete case.

Recall that for any $p$-complete ring $R$ we have a canonical map
\[\theta_R=\theta\colon A_{\inf}(R)\to R\]
which is the counit of the adjunction $W\dashv (-)^{\flat}$ between perfect algebras and $p$-complete rings.

We first recall several lemmas about $A_{\inf}$.

\begin{lemma}\label{inv_Ainf}
    Let $R$ be $p$-adically complete. An element $v\in A_{\inf}(R)$ is invertible if and only if $\theta(v)\in R$ is invertible.
    \begin{proof}
        We show the non-trivial direction. Write $v=\sum_{i}[v_i]p^i$ where $v_i\in R^{\flat}$. Suppose that the image $\theta(v)=\sum_i v_i^{\sharp}p^i$ is invertible, implying that $v_0^{\sharp}$ is invertible by $p$-completeness. Note that $v_0$ seen in $\varprojlim_{r\mapsto r^p} R$ has as first component $v_0^{\sharp}$. In this point of view, $v_0$ is a compatible system of $p$-th power roots of $v_0^{\sharp}$ this implies that in this perspective each component of $v_0$ is invertible and therefore that $v_0$ is invertible in $R^{\flat}$. By $p$-completeness of $A_{\inf}$, we conclude that $v$ is invertible.
    \end{proof}
\end{lemma}


\begin{lemma}\label{automatic_completensess}
Let $R$ be a ring with $pR=0$. Then if $x\in \ker(R^{\flat}\to R)$, $R^{\flat}$ is $x$-complete.
\begin{proof}Such an element $x$ can be written as $x=(x_i)_{i\geq 1}\in R^{\flat}$ such that $x_1=0$ and $x_i^p=x_{i-1}$ for $i\geq 2$. But then powers of $x$ are of the form 
\[(0,0,0,0,\dots,0,x_2,\dots)\]
implying $x$-completeness. Indeed, a Cauchy sequence will define by induction a unique element in $R^{\flat}$.
\end{proof}
    
\end{lemma}

We note the following automatic completeness property of $A_{\inf}$.

\begin{cor}\label{completeness_Ainf}
Let $R$ be $p$-complete. Then $A_{\inf}(R)$ is $(p,\xi)$-complete for any element $\xi\in \ker(\theta)$, where $\theta\colon A_{\inf}(R)\to R$ is the natural map.
\end{cor}
\begin{proof}To show this, as $A_{\inf}(R)$ is $p$-complete, this enough to check that $A_{\inf}(R)/pA_{\inf}(R)=R^{\flat}$ is $\xi$-complete, where the last lemma concludes.
\end{proof}

Now we can find a intrinsic definition of $p$-complete semiperfectoid rings, see (2) in \Cref{intrisic_semipfd} below. This is related to the equivalent definitions of quasiregular semiperfectoid rings in \cite[Definition 4.20 and Lemma 4.25]{bhatt2019Topologicala}.

\begin{proposition}[{cf.\ \cite[Lemma 3.5.8]{bhatt2025Aspects}}]\label{intrisic_semipfd}
    
Let $R$ be a $p$-adically complete ring. Then the following are equivalent.
\begin{enumerate}

    \item There is a perfectoid ring $S$ and a surjection $S\to R$.
    \item $R/pR$ is semiperfect and there is some $\varpi\in R$ such that $\varpi^p=pu$ for some $u\in R^{\times}$.
    \item $R/pR$ is semiperfect, there is a perfectoid algebra $S$ and a map $S\to R$.
    \item The canonical map 
    \[\theta\colon A_{\inf}(R)\to R\]
    is surjective and there is some distinguished element $\xi$ in the kernel, such that $R^{\flat}=W(R^{\flat})/(p)$ is $\xi$-complete.
\end{enumerate}
\end{proposition}
\begin{proof}We first show that (3) holds if and only if (4) does. If (4) holds, then take $\xi\in \ker(\theta)$ distinguished such that $R^{\flat}=W(R^{\flat})/(p)$ is $\xi$-complete. Then $A_{\inf}(R)/(\xi)$ is perfectoid (see \cite[Section 2.2.1]{cesnavicius2023purityflatcohomology}) and maps to $R$. Also quotienting by $p$ gives that $R^{\flat}\to R/pR$ is surjective, implying that $R/pR$ is semiperfect. If (3) holds then we have a commutative diagram
\[\begin{tikzcd}
A_{\inf}(S) \arrow[r] \arrow[d] & A_{\inf}(R) \arrow[d] \\
S \arrow[r]                    & R                   
\end{tikzcd}\]
implying that the image of the distinguished element in $A_{\inf}(S)$ which generates the kernel of the map $A_{\inf}(S)\to S$ is in the kernel of the map $A_{\inf}(R)\to R$. Quotienting by $p$, we get that $\xi$ is in the kernel of $R^{\flat}\to R/pR$ so \Cref{automatic_completensess} concludes.

Note that (4) implies (1) is clear, take $S=A_{\inf}(R)/(\xi)$. That (1) implies (2) also. We finally show that (2) implies (4). Consider the map
\[\theta\colon A_{\inf}(R)\to R.\]
As $R/pR$ is semiperfect, the above is surjective modulo $p$, therefore surjective by $p$-completeness. By \cite[Lemma 3.9]{bhatt2018Integral}, we can suppose that there is $\pi \in R^{\flat}$ with $\sharp(\pi)=\varpi'$ and $\varpi'=pu'$ for some unit $u'\in R^{\times}$.
Let $v$ be a lift of the unit $u'$. The element $v$ is a unit by \Cref{inv_Ainf}. Now $\xi=[\pi]-pv$ is distinguished and in the kernel -- also $\pi\in \ker(R^{\flat}\to R/pR)$ because this map is $\theta$ modulo $p$, which concludes by \Cref{automatic_completensess}.
\end{proof}

We also remark that we can classify Noetherian $p$-complete semiperfectoid rings.

\begin{lemma}\label{Noethsemiperfect}Let $R$ be a ring which is semiperfect and Noetherian. Then the natural maps $R^{\flat}\to R\to R_{\perf}$ are isomorphisms and those rings are a finite product of perfect fields.
    
\end{lemma}
\begin{proof}
By Noetherianity, let $N\in \N$ such that for every nilpotent $x\in R$, then $x^{p^N}=0$. Now let $x^{1/p^N}$ a $p^N$-th root of $x$. As this element is also nilpotent, we conclude that $x=(x^{1/p^N})^{p^N}=0$. Therefore, $R$ is perfect.
Note that a perfect Noetherian domain $S$ is a field. Indeed take $s\in S$ non-zero. Then the sequence of ideals $(s^{1/p^n})_{n\in \N}$ stabilizes, so without loss of generality we have $s=s^pt$ for some $t\in S$, and then $1=s^{p-1}t$, concluding.
So every prime of $R$ is maximal and as $R$ is Noetherian of Krull dimension zero and perfect, $R$ is a finite product of perfect fields.
\end{proof}

\begin{lemma}\label{charectization_pfd_charp}
    Let $R$ be a $p$-complete semiperfectoid ring such that the natural map
    \[R^{\flat}\to R/pR\]
    is injective. Then $pR=0$.
\end{lemma}
\begin{proof} Consider the canonical surjective map 
\[\theta\colon A_{\inf}(R)\to R,\]
and let $\xi=[\pi]+pu$ where $u\in A_{\inf}(R)$ is a unit and $\pi\in R^{\flat}=R/pR$ an element, such that $\xi\in \ker(\theta)$ as in \Cref{intrisic_semipfd}, (4).
 Now $[\pi]$ is sent to zero in $R/pR$, implying that $\pi=0$ in $R^{\flat}$ by the injectivity hypothesis. Now it concludes as $R=A_{\inf}(R)/pA_{\inf}(R)$.
    
\end{proof}

\begin{cor}\label{Noeth_pfd_rings}
   Let $R$ be a $p$-complete semiperfectoid ring such that $R/pR$ is Noetherian. Then $R$ is a finite product of perfect fields.
\end{cor}
    
\begin{proof} Combine \Cref{Noethsemiperfect} applied to $R/pR$ and \Cref{charectization_pfd_charp}.


    
\end{proof}





\subsection{Semiperfectoid rings and $p$-root closure}

As an application of arc descent, and more precisely the pullback square from \Cref{812prisms}, we show that we can recover that the perfectoidization of a $p$-torsion-free semiperfectoid ring $S$ is the $p$-completion of the $p$-root closure of $S$ in $S[1/p]$. This  revisits \cite[Theorem 1.3]{ishizuka2024Calculation} in a more concise way.


In general, if $A\subset B$ is a subring, then the \emph{$p$-root closure} of $A$ in $B$ is the subring defined    the union 
\[\bigcup_{i\in \N}C_i\]
where $C_0=A$ and $C_i\subset B$ is generated as a $C_{i-1}$-algebra by elements in $x\in B$ such that $x^p\in C_{i-1}$ for $i\geq 1$. This is a special case of the total $n$-root closure introduced in \cite{anderson1990Root}.

If $A$ is $p$-torsion-free ring, then the $p$-root closure of $A$ in $A[\frac{1}{p}]$ has actually a simpler description, as proved in \cite[Proposition 1]{roberts2008Root}. It is equal to
\begin{equation}\label{noaddneedRob}
    \{x\in A[\frac{1}{p}]\mid \exists m\in \N \quad x^{p^m}\in A\}
\end{equation}

What is non-evident is that the above is actually a ring. Namely if $x,y$ are in this set then $(x+y)^p=x^p+y^p+rp$ where $r\in A[\frac{1}{p}]$ is some element. This actually works as explained in \cite[Proposition 1]{roberts2008Root}.

We now make use \Cref{812prisms} to get that the perfectoidization  of a $p$-torsion-free semiperfectoid ring $S$ is the $p$-completion of the $p$-root closure of $S$ in $S[1/p]$ as in \cite[Theorem 1.3]{ishizuka2024Calculation}. These proofs are different. See \Cref{DiffRyoPaper} for the difference.

\begin{thm}
\label{semipfdproot}
Let $S$ be a $p$-torsion-free ring whose \(p\)-adic completion is a semiperfectoid ring. Then $S\to S_{\pfd}$ is identified to the natural map to the $p$-completion of the $p$-root closure of $S$ in $S[1/p]$.
    
\end{thm}
\begin{proof}
    Let's denote by $S\to \widetilde{S}$ the $p$-completion of the $p$-root closure of $S$ in $S[1/p]$. Now, using \Cref{812prisms}, and \cite[Proposition 2.1.8]{cesnavicius2023purityflatcohomology} which ensures that $\widetilde{S}$ is perfectoid, we get a pullback square
    \[\begin{tikzcd}
S_{\pfd} \arrow[r] \arrow[d] & \widetilde{S} \arrow[d]               \\
S/\sqrt{pS} \arrow[r]      & \widetilde{S}/\sqrt{p\widetilde{S}}
\end{tikzcd}\]

So to show the claim, it suffices to show that 
\[S/\sqrt{pS}\to \widetilde{S}/\sqrt{p\widetilde{S}}\]
is an isomorphism. Note that both these rings are perfect being semiperfect and reduced. Also, note that because $p=0$ in those rings we can disregard the $p$-adic completion on the ring on the right. Namely if $\overline{S}$ denotes the uncompleted $p$-root closure then $\overline{S}/p\overline{S}\to \widetilde{S}/p\widetilde{S}$ is an isomorphism, and the same then holds at reduction. So it suffices to show that
\[S/\sqrt{pS}\to \overline{S}/\sqrt{p\overline{S}}\]
is an isomorphism.

    Let $x\in S$ be an element which is sent to zero by the above map. Then there is some $M>0$ such that $x^{p^M}=py$ for some $y\in \overline{S}$. But then $x^{p^{N+M}}\in pS$ for some $N>0$ using \Cref{noaddneedRob}. Using that $S/\sqrt{pS}$ is perfect, it implies that the reduction modulo $\sqrt{pS}$ of $x$ is zero. To show surjectivity, let $y\in \overline{S}$. Because $y^{p^M}\in S$ for some $M$ and that these rings are perfect this shows that $y$ has a preimage. This concludes.
\end{proof}

\begin{rem} \label{DiffRyoPaper}
   Let us explain the difference between the proof of \Cref{semipfdproot} given here and the proof of \cite[Theorem 1.3]{ishizuka2024Calculation}.
    In this paper, the key ingredients are the pullback square (\Cref{812prisms}) proved in \cite{bhatt2022Prismsa} and the fact that the \(p\)-adic completion of the \(p\)-root closure of a semiperfectoid ring is perfectoid proved in \cite[Proposition 2.1.8]{cesnavicius2023purityflatcohomology}. The proof of \Cref{semipfdproot} is more direct and conceptual than the proof of \cite[Theorem 1.3]{ishizuka2024Calculation}.

    On the other hand, in \cite{ishizuka2024Calculation}, the first-named author of this paper first proved a criterion for such isomorphisms (\cite[Theorem 5.7]{ishizuka2024Calculation}) in which he assumes the \(p\)-adic completion of the \(p\)-root closure becomes perfectoid. When applying this criterion for the semiperfectoid case, \cite[Proposition 2.1.8]{cesnavicius2023purityflatcohomology} is used too. In contrast, the arc topology is not used.
    The proof only uses the universality of the perfectoidization. A relation to uniform completion is also studied. 
\end{rem}

\begin{rem}\label{whathappensactually}
    Both taking the $p$-root closure and $p$-completing are essential in the process of perfectoidization of semiperfectoid rings.
    We demonstrate on a simple example how this behaves. Let $R$ be a perfectoid ring. Consider the semiperfectoid ring
    \[S=R\abracket{x^{1/p^{\infty}}}/(x)\]
    Then inverting $p$ gives
    \[R\left[\frac{1}{p}\right]\abracket{x^{1/p^{\infty}}}/(x).\]
    Note that any element in $S[1/p]$ which is the image of an element in $(x^{1/{p^\infty}})R[\frac{1}{p}][x^{1/p^{\infty}}]$ is in the $p$-root closure of $S$ in $S[1/p]$. Indeed, these elements are nilpotent. But note also that all these elements are infinitely $p$-divisible: if $\epsilon$ is such an element then $\epsilon/p^n$ is also such an element for any $n\in \N$. Therefore, they will be killed in the $p$-completion. An intermediate step to the process of taking $p$-root closure and then $p$-completion will therefore at least be taking the quotient by
    \[S'\defeq R\abracket{x^{1/p^{\infty}}}/\overline{(x^{1/p^{\infty}})}\cong R.\]
    As this ring is perfectoid, the quotient map $S\to S'$ has to be the perfectoidization, so we conclude that the above was in fact describing the whole process.

    The takeway in this case is as follows: nilpotent elements become infinitely $p$-divisible in the $p$-root closure, and therefore zero in the $p$-completion.
\end{rem}

\section{Perfectoid ideals}\label{pfdideals_section}

In this section we proof some useful lemmas about perfectoid ideals.
\begin{rem}\label{def_pfdization_ideal}
    Recall that for any ideal $I\subset R$ then there is a initial perfectoid ring $(R/I)_{\pfd}$ with an universal surjective map $R\to (R/I)_{\pfd}$ which is initial with respect to maps to perfectoid rings. See \cite[Remark 1.13]{bhatt2022Prismsa}. To be really precise, we can first consider the kernel $\overline{I}$ of the surjective map $R\to (R/I)^{p\wedge}$ to the classical $p$-completion of $R/I$. As perfectoid rings are all classicaly $p$-complete, this is anyway an intermediate step to the perfectoidization. This ideal is the closure in $R$ of $I$ for the $p$-adic topology on $R$. Now $\overline{I}$ is a derived $p$-complete ideal so that we can apply \cite[Remark 1.13]{bhatt2022Prismsa} as written.
\end{rem}

\begin{definition}[{cf.\ \cite[Definition 2.1 and Lemma 2.2]{Dine_2024}}]\label{pfd_ideal}Let $R$ be a perfectoid ring. An ideal $I\subset R$ is said to be \emph{perfectoid} if $I=I_{\pfd}$.
    
\end{definition}

\begin{rem}\label{rad_and_p_comp} Because perfectoid rings are $p$-complete and reduced \cite[Section 2.1.2 and Section 2.1.3]{cesnavicius2023purityflatcohomology}, any perfectoid ideal is $p$-complete and radical.
    
\end{rem}

Note for example the following property, which implies that the perfectoidization of a non-zero semiperfectoid ring is non-zero.

\begin{lemma}\label{proper_Ipfd}
Let $R$ be a perfectoid ring. 
\begin{enumerate}
    \item Let $\mk$ be a maximal ideal. Then $R/\mk$ is perfect in characteristic $p$.
    \item If $I\subset R$ is a proper ideal, then $I_{\emph{pfd}}$ is also proper.
\end{enumerate}
    \begin{proof}
        For the first assertion, note that as $\mk$ is maximal, it is closed. The closure of an ideal is again ideal. If $\mk$ is not closed, then $\overline{\mk}=R$. Now if $1$ was a limit of elements in $\mk$, there would be an element $m\in \mk$ with $1-m\in pR$ and $m$ would be invertible, a contradiction. Therefore $R/\mk$ is a $p$-adically complete field, which implies that it is a field of characteristic $p$. Because $R$ is perfectoid, it is also semiperfect, therefore perfect.

        If $I$ is proper, then $I\subset \mk$ for some maximal ideal $\mk$. But then $I_{\text{pfd}}\subset \mk_{\text{pfd}}=\mk$.
    \end{proof}
\end{lemma}

We will now prove \Cref{lemma:ideals} which is a really useful lemma about perfectoid ideals. Before, we simply recall the following fact.

\begin{lemma}\label{tiltcomletion} Let $B$ be a perfect algebra in characteristic $p$. Let $\pi\in B$ be an element. Then we have a natural isomorphism
    \begin{equation*}
        \left(B/\pi B\right)^{\flat}\to \left(B \right)^{\wedge,\pi}.
    \end{equation*}
\end{lemma}
\begin{proof}
The morphism of diagrams below is an isomorphism.
\[
    \begin{tikzcd}
{} \arrow[r] & B/\pi B \arrow[r, "(-)^p"] \arrow[d, "(-)^{p^{3}}"] & B/\pi B \arrow[r, "(-)^p"] \arrow[d, "(-)^{p^{2}}"] & B/\pi B \arrow[r, "(-)^p"] \arrow[d, "(-)^p"] & B/\pi B \arrow[d, "="] \\
{} \arrow[r] & B/\pi^{p^3}B \arrow[r]                             & B/\pi^{p^2}B \arrow[r]                             & B/\pi^p B \arrow[r]                           & B/\pi B               
\end{tikzcd}
\]
\end{proof}

The following is to compare to \cite[Lemma 3.3 and Theorem 3.5]{Dine_2024} as well as \cite[Proposition 2.1.11 (c)]{cesnavicius2023purityflatcohomology}. 

\begin{lemma}\label{lemma:ideals}
Let $R$ be a perfectoid ring and $I\subset R$ be an ideal. Denote by $\theta\colon A_{\inf}(R)\to R$ the natural surjection, with kernel generated by $\xi$. Let $\varpi^{\flat}\in R^{\flat}$ be the image of $\xi$ modulo $p$ in $R^{\flat}$. Let also $\varpi=\sharp(\varpi^{\flat})$. Then the following are equivalent.
\begin{enumerate}
    \item The ring $R/I$ is perfectoid.
    \item The ideal $I$ is $p$-adically generated by $\sharp(I^{\flat})$. In other words we have $I=\overline{(\sharp(I^{\flat}))}$. More precisely, we have $I=\theta(W(I^{\flat}))$.
    \item There exists a subset $S\subset I$ of elements that admits compatible system of $p$-power roots such that
\[I=\overline{(s^{1/p^{\infty}})_{s\in S}}.\]
\end{enumerate}
\begin{proof}We will show 
\[(1)\implies (2)\implies (3)\implies  (1).\]
Note that $(2)\implies (3)$ is immediate.
\begin{itemize}
    \item $(1)\implies (2).$ If we suppose that $R/I$ is perfectoid, then we have a surjection
    \[A_{\inf}(R)\to A_{\inf}(R/I)\]
    because we can check that this map is surjective modulo $\xi$ by $\xi$-completeness.  Therefore the map 
    \[R^{\flat}\to (R/I)^{\flat}\]
    is surjective being the reduction modulo $p$ of the above map. 

 One checks using the functoriality of $(-)^{\flat}$ seen as compatible system of $p$-power roots\footnote{Which is possible because both rings are $p$-adically complete.} that (one takes this equality as a definition of the tilt of an ideal)
 \begin{equation}\label{titl_ideal_def}
     \ker(R^{\flat}\to (R/I)^{\flat})=\{(r_n)\in R^{\flat} \mid r_n\in I \quad \forall n\in\N\}=I^{\flat}.
 \end{equation}
   In other words the canonical map $R^{\flat}/I^{\flat}\to (R/I)^{\flat}$ is an isomorphism. Now we see that the kernel of the surjective map
    \[A_{\inf}(R)\to A_{\inf}(R/I)\]
 is equal to $W(I^{\flat})$ by inspecting the functoriality of the Witt vectors. Going modulo $\xi$ now shows the claim.
    \item $(3)\implies (1)$. Say that for any $s\in S$ we denote by $x_s\in I^{\flat}$ a choice of compatible system of $p$-power roots of $s$. As $R$ is perfectoid, the map
    \[\sharp\colon R^{\flat}/\varpi^{\flat}R^{\flat}\to R/\varpi R\]
    is an isomorphism. For $s\in S$, the pre-image under this isomorphism $\sharp$ of 
    \begin{equation*}
        s\mod \varpi \in R/\varpi R
    \end{equation*}
    is 
    \begin{equation*}
        x_s \mod \varpi^{\flat}\in R^{\flat}/\varpi^{\flat}R^{\flat}.
    \end{equation*}
    Therefore the preimage by $\sharp$ of the image of $I$ in $R/\varpi R$ is given by the ideal generated by the image of the ideal $J:=(x_{s}^{1/p^n})_{n\geq 0,s\in S}$ in $R^{\flat}/\varpi^{\flat}R^{\flat}$. Therefore we see that
    \[R^{\flat}/(J,\varpi^{\flat})\to R/(I,\varpi)\]
    is an isomorphism. Now, because $R^{\flat}/J$ is perfect, using \Cref{tiltcomletion} and the above, we can identify $(R/I)^{\flat}=(R/(I,\varpi))^{\flat}$ to
    \[R^{\flat}/\overline{(J)}\]
    where the closure is taken in the $\varpi^{\flat}$-adic topology. As a consequence, the map
    \[R^{\flat}\to (R/I)^{\flat}\]
    is surjective. As a consequence $A_{\inf}(R)\to A_{\inf}(R/I)$ is also surjective.
    
    
    Note also that by construction of $J$, we have that $\theta\colon W\left(\overline{J}\right)\to I$ is surjective. We now use the snake lemma on the diagram below to deduce that $\xi A_{\inf}(R)=\ker(\theta_R)\to \ker(\theta_{R/I})$ is surjective, concluding.
    \[\begin{tikzcd}
            &                                     & \ker(\theta_R) \arrow[d] \arrow[r] & \ker(\theta_{R/I}) \arrow[d] \arrow[llddd, dashed] &   \\
0 \arrow[r] & W(\overline{J})\arrow[r] \arrow[d] & A_{\inf}(R) \arrow[r] \arrow[d]    & A_{\inf}(R/I) \arrow[d] \arrow[r]                  & 0 \\
0 \arrow[r] & I \arrow[r] \arrow[d]               & R \arrow[r]                        & R/I \arrow[r]                                      & 0 \\
            & 0                                   &                                    &                                                    &  
\end{tikzcd}\]
\end{itemize}

\end{proof}
    
\end{lemma}

    

We note the following corollary.

\begin{cor}[{cf. \cite[Proposition 2.5]{fayolle2025Centers}}]\label{Pfdizationideal_generated_by_p_power roots}Let $R$ be a perfectoid ring and $I=(i_{\alpha})_{\alpha\in A} \subset R$ be an ideal generated by elements $i_{\alpha}$ which admits a compatible system of $p$-power roots.
Then $I_{\pfd}=\overline{(i_{\alpha}^{1/p^{\infty}})}$.
\end{cor}
\begin{proof}By \Cref{lemma:ideals}, (3), we see that $\overline{(i_{\alpha}^{1/p^{\infty}})}$ is a perfectoid ideal. Therefore $I_{\pfd}\subset \overline{(i_{\alpha}^{1/p^{\infty}})}$. But also, as $I_{\pfd}$ is radical and $p$-complete (\Cref{rad_and_p_comp}), we have $\overline{(i_{\alpha}^{1/p^{\infty}})}\subset I_{\pfd}$.
    
\end{proof}

We will see in \Cref{TorsionExample} that $(px)_{\pfd}=(p)_{\pfd}\cap (x)_{\pfd}$. We note that this is not a surprise because of the following more general phenomenon which is a direct consequence of $p$-complete arc descent.

\begin{lemma}\label{intersection_product_pfdization}
Let $R$ be a perfectoid ring and $I,J\subset R$ two ideals. Then
\[(IJ)_{\pfd}=(I\cap J)_{\pfd}=I_{\pfd}\cap J_{\pfd}\]
   In particular, the intersection of two perfectoid ideals is again perfectoid. 
\end{lemma}
\begin{proof}
    As $(I\cap J)^2\subset IJ$, note that $R/IJ\to R/I\cap J$ induces an isomorphism on (co)representing functors in the ($p$-complete) arc topology -- the first claim now follows from \cite[Corollary 8.11]{bhatt2022Prismsa}. As for the second claim, note that $R/I\cap J\to R/I\times R/J$ is a ($p$-complete) arc cover, which implies that we have an injection at perfectoidization $R/(I\cap J)_{\pfd}\to R/I_{\pfd}\times R/J_{\pfd}$ by \Cref{archull}. This implies now that the kernel of the surjection $R\to R/(I\cap J)_{\pfd}$ is both $(I\cap J)_{\pfd}$ and $I_{\pfd}\cap J_{\pfd}$, concluding.
\end{proof}

\begin{rem}We wrote that $I\cap J$ is perfectoid if $I$ and $J$ are as a conclusion of the lemma because it is a consequence of the above proof. But note that the statement holds more generally for arbitrary intersections of families  $(I_{\alpha})_{\alpha\in A}$ of perfectoid ideals because the semiperfectoid $R/\bigcap_{\alpha\in A}I_{\alpha}$ embeds into the perfectoid $\prod_{\alpha\in A}R/I_{\alpha}$. See the argument in \cite[Lemma 2.3]{Dine_2024}.
    
\end{rem}

We thank Dimitri Dine for pointing out that \Cref{intersection_product_pfdization} gives a slick proof of \cite[Proposition 4.7]{Dine_2024} that we take the opportunity to display here.

\begin{proposition}\label{prop:slick_domain_proof}Let $R$ be a perfectoid ring and $I\subset R$ a perfectoid ideal. Denote by $I^{\flat}$ the tilt of the ideal in $R^{\flat}$ (see \Cref{titl_ideal_def}). Then $I$ is prime if and only if $I^{\flat}$ is.
\end{proposition}
\begin{proof}
Because the $R$-algebra untilt of $R^{\flat}/I^{\flat}$ is $R/I$ it suffices to treat the case $I=0$ and to show that $R$ is a domain if and only $R^{\flat}$ is a domain. If $R$ is a domain, then $R^{\flat}$ is also which is directly seen using the identification $R^{\flat}=\varprojlim_{(-)^p}R$. We now prove the converse. Assume that $R^{\flat}$ is a domain. Let $I$ and $J$ be two ideals such that $IJ=0$. Then by \Cref{intersection_product_pfdization} we have $I_{\pfd}\cap J_{\pfd}=I_{\pfd}J_{\pfd}=0$. By the ideal tilt-untilt correspondence (see \cite[Main Theorem A]{Dine_2024} and \cite[Remark 3.12]{navarrochafloque2026perfectoid}) we have $I_{\pfd}^{\flat}\cap J_{\pfd}^{\flat}=I_{\pfd}^{\flat}J_{\pfd}^{\flat}=0$. As $R^{\flat}$ is a domain we get that either $I^{\flat}$ or $J^{\flat}$ is zero. Therefore, using the correspondence again, either $I_{\pfd}$ or $J_{\pfd}$ are zero, which concludes as then either $I$ or $J$ is zero.
\end{proof}

We also note the following interesting corollary of \Cref{lemma:ideals}.

\begin{lemma}\label{pfd_ideals_sum}Let $R$ be a perfectoid ring.
\begin{enumerate}
    \item Let $A$ be a set and $(I_{\alpha})_{\alpha\in A}$ be a collection of perfectoid ideals of $R$. Then the $p$-adic closure in $R$ of the sum $\Sigma_{\alpha\in A}I_{\alpha}$ is perfectoid.
    \item Let $I\subset R$ be an ideal which is closed for the $p$-adic topology in $R$. Then $I$ is a perfectoid ideal if and only for any $i\in I$, then $(i)_{\pfd}\subset I$. 
    \item A finite sum of perfectoid ideals is again perfectoid (cf. \cite[Proposition 2.7]{fayolle2025Centers}).
\end{enumerate}
\end{lemma}
\begin{proof}For the first statement, note that by \Cref{lemma:ideals} (3), then there exists sets $S_{\alpha}$ that we can suppose closed up to taking $p$-power roots (in particular every element of $S_{\alpha}$ admits a system of compatible $p$-power roots) such that $\overline{(S_{\alpha})}=I_{\alpha}$, and this for every $\alpha\in A$. Now, it follows that the the $p$-adic closure of the sum $\Sigma_{\alpha\in A}I_{\alpha}$ is $\overline{(\bigcup_{\alpha\in A}S_{\alpha})}$ which concludes using \Cref{lemma:ideals} (3), again. 

For the second point, first one sees that $I$ is a perfectoid ideal then the claimed property holds true by universal property. The converse is a direct application of the first point: In this case $I$ is the $p$-adic closure of $\sum_{i\in I}(i)_{\pfd}$.

For the third point, note that building on the above a non-trivial input is that we can ignore $p$-completion. Note that $I\cap J$ is a perfectoid ideal (see \Cref{intersection_product_pfdization}). It suffices to prove the statement for two perfectoid ideals $I$ and $J$ by induction. This is a direct consequence of arc descent using \Cref{812prisms} for the map $R/I\cap J\to R/J$ which is an isomorphism away from $I$, and that the square 
\[\begin{tikzcd}
R/I\cap J \arrow[d] \arrow[r] & R/J \arrow[d] \\
R/I \arrow[r]                 & R/I+J    
\end{tikzcd}\]
is a pullback square with surjective arrows, so that the bottom right of the pullback square (also with surjective arrows)
\[\begin{tikzcd}
R/I\cap J \arrow[d] \arrow[r] & R/J \arrow[d] \\
R/I \arrow[r]                 & (R/I+J)_{\pfd}    
\end{tikzcd}\]
has to coincide with the bottom right of this one.
\end{proof}

We now prove a result about how perfectoid ideals extend along extension by a perfectoid ring.
This is proved in the case of principal ideals in \cite[Lemma 2.4.3]{cai2023Perfectoid} and finitely generated ideals in \cite[Proposition 2.9]{fayolle2025Centers}.

\begin{lemma}[{cf.\ \cite{cai2023Perfectoid} and \cite{fayolle2025Centers}}]\label{pfdideals_along_fflatextension} 
Let $R$ be a perfectoid ring and $R\to R'$ a map to a perfectoid ring. Let $I\subset R$ be a perfectoid ideal in the sense that $R/I$ is a perfectoid ring. Then the natural map
\[I\widehat{\otimes}_R^{L}R'\to (IR')_{\pfd}\]
is an isomorphism. In particular $(IR')_{\pfd}$ is the $p$-adic closure of $IR'$ in $R'$.
\end{lemma}
\begin{proof} As $R$, $R/I$ and $R'$ are all perfectoid $R/I\widehat{\otimes}_R^{L}R' $ is discrete by \cite[Lecture IV, Proposition 2.11]{Bhatt2018Eilenberg}.


    Note also that by compatibility of perfectoidization and and base change (\cite[Proposition 8.13]{bhatt2022Prismsa})
    \[R'/(IR')_{\pfd}=\left(R/I\widehat{\otimes}_R^{L} R'\right)_{\pfd}=R/I\widehat{\otimes}_R^{L}R' \]
    as every term in the tensor product is perfectoid. The first equality is obtained using the universal property of this quotient as the initial perfectoid $R'$-algebra where $IR'$ is zero. Now the exact sequence
    \[0\to I\to R\to R/I\to 0\]
    gives after derived $p$-complete base change to $R'$ a fiber sequence in $\widehat{\Ds}(R')$
    \[I\widehat{\otimes}_R^{L}R'\to R'\to R'/(IR')_{\pfd}.\]
    But has the two terms on the right are discrete and that the map is surjective, we conclude that the first term is also discrete being equal to the kernel of the map between these discrete modules, concluding.
\end{proof}

\begin{rem}\label{rem:flat_ideals}The argument in \cite[Lecture IV, Proposition 2.11]{Bhatt2018Eilenberg} ultimately relies on the fact that if $R$ is a perfect ring in characteristic $p$ and $f\in R$ is an element, then $(f^{1/p^{\infty}})$ is flat as an $R$-module. We can generalize this statement to perfectoid rings, see \Cref{prop:pfd_ideal_flatness}.
\end{rem}

\begin{rem}\label{rem:Bhatt_generalize_exmaple}
 The second statement in \Cref{prop:pfd_ideal_flatness} is a generalization of \cite[Example 3.2.13]{bhatt2025Aspects}. Namely the latter deals with the case where the elements generating the ideal admit a $p$-compatible sequence of $p$-power roots.
\end{rem}

\begin{proposition}\label{prop:pfd_ideal_flatness}Let $R$ be a perfectoid ring and $I\subseteq R$ an ideal.
\begin{enumerate}
    \item Suppose that $I=\sqrt{xR}$ for an element $x\in R$. Then $I_{\pfd}$ is a $p$-completely flat $R$-module.
    \item If $I=\sqrt{(x_1R+\dots+x_nR)}$ then the $p$-complete  tor dimension of $R/I_{\pfd}$ as a $R$-module is at most $n$.
\end{enumerate}
    
\end{proposition}

\begin{proof}We first prove (1). We can assume that $I=xR$ as two ideals with the same radical have the same perfectoidization. Using that $p$-complete flatness can be checked after derived $p$-completely faithfully flat base change, \Cref{pfdideals_along_fflatextension} and André's Lemma \cite[Theorem 7.14]{bhatt2022Prismsa} we can assume that $R$ is absolutely integrally closed. Now, choose a compatible of sequence of $p$-power roots of $p$ and $x$ in $R$. This leads to a map
\[\widetilde{R}:=\Z_p[p^{1/p^{{\infty}}},X^{1/p^{\infty}}]^{p\wedge}\to R\]
sending $X$ to $x$. Now, it suffices to prove that $(X\widetilde{R})_{\pfd}=\overline{(X^{1/p^{\infty}})}$ is a $p$-completely flat $\widetilde{R}$-module using that the derived $p$-complete base change to $R$ would be again $p$-completely flat and \Cref{pfdideals_along_fflatextension} again.

The uncompleted quotient $\widetilde{R}/(X^{1/p^{\infty}})$ is $p$-torsion free: namely if 
$$p\left(\sum_{k\in \N[1/p]}a_kX^k\right)=\sum_{k\in \N[1/p]_{\geq 1/p^m}}b_kX^k$$
for some $m\in \N$, then using that $\widetilde{R}$ embeds into the $p$-torsion free module $\prod_{\N[1/p]}\Z_p[p^{1/p^{\infty}}]^{p\wedge}$ by the coefficients of the power series we get our claim. Therefore taking the derived $p$-completion $\Lambda_p(-)$ we get a fiber sequence in $\widehat{\Ds}(\widetilde{R})$
\[\Lambda_p(X^{1/p^{\infty}})\to \widetilde{R}\to \Lambda_p\left(\widetilde{R}/(X^{1/p^{\infty}})\right)\]
Note that both the right and the left term are actually discrete classical completions because both modules that we are taking the derived $p$-completion of are $p$-torsion free. Namely, the left term is $\Lambda_p\left(\widetilde{R}/(X^{1/p^{\infty}})\right)=\widetilde{R}/\overline{(X^{1/p^{\infty}})}$ and therefore the closure $\overline{(X^{1/p^{\infty}})}$ in $\widetilde{R}$ is equal to the (derived) $p$-completion of the module $(X^{1/p^{\infty}})$.
Now $(X^{1/p^{\infty}})$ is a flat $\widetilde{R}$-module as the union of rank 1 free modules $(X^{1/p^n})$ for $n\in \N$. Therefore $(X\widetilde{R})_{\pfd}=\overline{(X^{1/p^{\infty}})}$ is $p$-completely flat as the $p$-completion of a $p$-torsion free flat module (see \cite[Lemma 4.7.(2)]{bhatt2019Topologicala} for example).

As for the second assertion, we proceed using an induction on the number of generators of $I$ up to taking the radical, the case of one generator being handle by the first statement of the lemma.



    
\end{proof}

\begin{rem}\label{rem:striking} Note that putting together \Cref{prop:pfd_ideal_flatness} and \cite[Proposition 4.15]{navarrochafloque2026perfectoid} is quite striking: namely, taking the notation of the latter $(d)_{\pfd}=\overline{(\sharp(d_n)^{1/p^\infty})_{n \geq 0}}$ is $p$-completely flat, which in principle puts strong constrains on the relations that the generators can have.
\end{rem}


\section{Study of torsion in perfectoid rings}\label{torsion_section}

The main results of this section show that torsion in perfectoid rings is controlled, see \Cref{torsion_is_almostzero} and \Cref{bounded_d_infinity_torsion}.

\subsection{Technical results on $p$-completely faithfully flat base change}\label{tech_pcom_subsec}

As a prerequisite for \Cref{tame_torsion_section} and for future reference, we prove some general lemmas on $p$-completely faithfully flat base change.

We first mention the following lemma which mention properties of $p$-completely flat base change.

\begin{lemma} \label{CompBaseChange}
    Let \(R \to R'\) be a \(p\)-completely flat morphism.
      \begin{enumerate}
          \item Let $M$ be an $R$ module of bounded $p^{\infty}$-torsion. Then the derived $p$-completed base change of $M$ to $R'$ is discrete and correspond to the classical completion of the underived tensor product $M\otimes_R R'$ \emph{(cf.  \cite[Lemma 2.3.5]{cai2023Perfectoid}}.
          \item Let \(M\) be any derived $p$-complete $R$-module. The derived $p$-completed base change of $M$ to $R'$ is discrete if and only if for some presentation by a quotient $\coker(P\subset N)=M$ where $N$ is a module of bounded $p^{\infty}$-torsion, the map $P\subset N$ stays injective after derived $p$-completed base change to $R'$.
      \end{enumerate} 
\end{lemma}

\begin{proof}
   
   First, we will prove that the derived completed tensor product \(M \widehat{\otimes}^L_R R'\) is the same as the \(p\)-adic completion of the tensor product \(M \otimes_R R'\) for any $p$-torsion-free \(R\)-module \(M\). We denote the derived completion functor by $\Lambda_p$. We have,
    \begin{align*}
        &\Lambda_p(M\otimes_R^L R')=R\varprojlim_{n\geq 1} (M\otimes_R^L R'\otimes_{R'} R'/p^nR')=R\varprojlim_{n\geq 1} (M\otimes_R^L R/p^nR\otimes_{R/p^nR} R'/p^nR') \\
        &=R\varprojlim_{n\geq 1}(M/p^nM\otimes_{R/p^nR}R'/p^nR')=\varprojlim_{n\geq 1}(M/p^nM\otimes_{R/p^nR}R'/p^nR')
    \end{align*}
    where the first equality is a definition, the second one uses that $A\otimes_A^L -$ is the identity for any ring $A$, the third one holds because $M$ is $p$-torsion-free and $R/p^nR\to R'/p^nR'$ is flat. The fourth one is because the system is Mittag-Leffler. Now the last one is indeed the \(p\)-adic completion of the tensor product \(M \otimes_R R'\).
    
    Now suppose that $M$ is bounded $p^{\infty}$-torsion. Take the exact sequence
    \begin{equation*}
        0 \to M[p^\infty] \to M \to M/M[p^\infty] \to 0
    \end{equation*}
    of \(R\)-modules.
    Taking the $p$-completed derived base change, we have a fiber sequence
    \begin{equation*}
        (M[p^\infty]) \widehat{\otimes}^L_R R' \to M \widehat{\otimes}^L_R R' \to M/M[p^\infty] \widehat{\otimes}^L_R R'
    \end{equation*}
    in \(\widehat{\mcalD}(R')\). But note that because $M[p^{\infty}]$ is bounded torsion this is a $R/p^NR$ module for some $N\in \N$. So $ (M[p^\infty]) \widehat{\otimes}^L_R R'=(M[p^\infty]) \widehat{\otimes}^L_{R/p^NR} R'/p^NR'$ which is concentrated in degree zero by flatness. This and that $M/M[p^\infty] \widehat{\otimes}^L_R R'$ is concentrated in degree zero by the above, shows that the middle term is also in degree zero.
   

    We will prove the second statement.
    Take a short exact sequence \(0 \to P \to N \to M \to 0\) where $N$ and therefore $P$ as a submodule, are bounded $p^{\infty}$-torsion \(R\)-modules.
    Then we have a fiber sequence
    \begin{equation*}
        P \widehat{\otimes}^L_R R' \to N \widehat{\otimes}^L_R R' \to M \widehat{\otimes}^L_R R'
    \end{equation*}
    in \(\widehat{\Ds}(R')\).
   As the first two terms are in degree zero by (1), the last term is concentrated in degree \(0\) if and only if the first map is injective.
\end{proof}

For example, one can deduce the following using \Cref{CompBaseChange}.

\begin{lemma}\label{torsion_and_flatness} Let $R\to R'$ be a $p$-completely flat map between derived $p$-complete rings. Suppose that $R$ is of bounded $p^{\infty}$-torsion.
  Suppose that $d\in R$ is a non-zero divisor. If $R/dR$ is of bounded $p^{\infty}$-torsion, then $d\in R'$ is also a non-zero divisor. In particular if $R$ is $p$-torsion-free, so is $R'$.

\end{lemma}
\begin{proof}The map $R\xrightarrow{\cdot d} R$ stays injective after derived $p$-completed base change by $R'$ using \Cref{CompBaseChange} (2) because the cokernel is discrete by \Cref{CompBaseChange} (1).
\end{proof}

Note the following.

\begin{lemma}\label{fflat_classpcomplete}Let $R\to R'$ be a $p$-completely faithfully flat map between $p$-adically complete rings. Then $R\to R'$ is injective. In particular any $p$-completely faithfully flat maps between perfectoid rings is injective.
    
\end{lemma}
\begin{proof}By definition that $R'$ is $p$-completely faithfully flat implies that $R'/^Lp^nR'$ is faithfully flat over $R/p^nR$. Therefore it is discrete and equal to $R'/p^nR'$.

We can write $R\to R'$ as a limit of maps
\[\varprojlim_{n\geq 1} R/p^n R\to \varprojlim_{n\geq 1}R'/p^nR'\]
all of which are faithfully flat, therefore injective. Now any such limit of injective maps is injective.
\end{proof}

Aiming to \Cref{PerfectoidCheckFFdescent} as a consequence we prove the lemma below.

\begin{lemma} \label{pFFlatDescentLemma}
    Let \(R\) be a discrete ring and let \(S \to S'\) be a map of \(E_{\infty}\)-\(R\)-algebras.
    Take a \(p\)-completely faithfully flat map \(R \to R'\) of derived \(p\)-complete rings.
    Then the map \(S \to S'\) is an isomorphism if and only if the \(p\)-completed derived base change \(S \widehat{\otimes}^L_R R' \to S' \widehat{\otimes}^L_R R'\) is an isomorphism.
\end{lemma}

\begin{proof}
    First note that \(R' \otimes^L_R R/pR\) is discrete and \(R'/pR'\) is faithfully flat over \(R/pR\). In particular, \(R'\) is also a discrete \(R\)-algebra by derived Nakayama lemma.
    The only if direction is obvious. Assume the converse. Set \(C\) be the cofiber of the map \(S \to S'\) in \(\mcalD(R)\).
    Since derived \(p\)-completeness is stable under limits, this cofiber \(C\) is also derived \(p\)-complete.
    Moreover the assumption implies \(C \widehat{\otimes}^L_R R' = 0\).
    Taking \(- \otimes^L_R R/pR\), we have isomorphisms in \(\mcalD(R/pR)\):
    \begin{equation*}
        (C \widehat{\otimes}^L_R R') \otimes^L_R R/pR \cong C \otimes^L_R R'/pR' \cong (C \otimes^L_R R/pR) \otimes^L_{R/pR} R'/pR' = 0.
    \end{equation*}
    By the faithfully flatness, the derived tensor product \(C \otimes^L_R R/pR\) vanishes.
    Again using derived Nakayama lemma, \(C\) is also zero and thus \(S \to S'\) is an isomorphism.
\end{proof}

\begin{lemma}
\label{PerfectoidCheckFFdescent} Let $R$ be a perfectoid ring. 
Let $S$ be a derived $p$-complete $R$-algebra. Let $R\to R'$ a $p$-completely faithfully flat map where $R'$ is a perfectoid ring. Then $S$ is perfectoid if and only if $S'\defeq S\widehat{\otimes}_{R}R'$ is perfectoid.
\end{lemma}

\begin{proof}
    If $S$ is perfectoid, then $S'$ is also by \cite[Proposition 2.1.11 (b)]{cesnavicius2023purityflatcohomology} or \cite[Proposition 8.13]{bhatt2022Prismsa}.

Conversely if $S'$ is perfectoid, we have the following pushout squares of $E_{\infty}$-rings
\[\begin{tikzcd}
R_{\pfd}=R \arrow[r] \arrow[d] & S \arrow[d] \arrow[r] & S_{\pfd} \arrow[d] \\
R'_{\pfd}=R' \arrow[r]         & S' \arrow[r, "\cong"] & S'_{\pfd}         
\end{tikzcd}\]
by \cite[Proposition 8.13]{bhatt2022Prismsa}. Therefore the map $S\to S_{\pfd}$ is an isomorphism after the $p$-completely faithfully-flat base change by $- \widehat{\otimes}^L_R R'$, which implies the claim (\Cref{pFFlatDescentLemma}).
\end{proof}

\begin{rem}
\Cref{PerfectoidCheckFFdescent} sheds another point of view of the fact that if $R$ is a perfectoid ring and $R\to S$ is a finite étale map, then $S$ is perfectoid. 
Say $R\to R'$ is $p$-completely faithfully flat where $R'$ is absolutely integrally closed \cite[Theorem 7.14]{bhatt2022Prismsa}. Up to a further Zariski-refinement of $R'$, we can suppose that the finite \'etale extension $S'\defeq S\widehat{\otimes}_R R'$ of \(R'\) is isomorphic to a finite product $R'\times \cdots \times R'$.
Now, this product is perfectoid, see for example \cite[Proposition 2.1.11 (d)]{cesnavicius2023purityflatcohomology}. Then use that the perfectoid property can be checked up to $p$-completely faithfully flat base change as explained in \Cref{PerfectoidCheckFFdescent} to conclude.
\end{rem}

As an application of \Cref{pfdideals_along_fflatextension}, we show that $p$-completely (faithfully) flat maps are compatible with the decomposition of a perfectoid ring into a $p$-torsion free perfectoid ring and a $p$-torsion perfect ring (as in \cite[Section 2.1.3]{cesnavicius2023purityflatcohomology}).

\begin{proposition}
    \label{Structure_flatmaps_pfd} Let $R\to R'$ be a $p$-completely (faithfully) flat map between perfectoid rings. Then, if we write 
\[R\to T\times_{\overline{T}}\overline{R} \quad R'\to T'\times_{\overline{T'}}\overline{R'},\] 
for the decompositions of $R$ and $R'$ into $p$-torsion free perfectoid rings \(T\) and \(T'\) and their reduced modulo $p$-fibers (see \cite[Section 2.1.3]{cesnavicius2023purityflatcohomology}), then the induced maps $T\to T'$ and $\overline{R}\to \overline{R'}$ are also $p$-completely (faithfully) flat.
\end{proposition}
\begin{proof}

It suffices to show that the derived $p$-complete base change of $R\to R'$ along $T\widehat{\otimes}^{L}_R (-)$ and $\overline{R}\widehat{\otimes}^{L}_R (-)$ are respectively $T\to T'$ and $\overline{R}\to \overline{R'}$.

First, note that as $T$ is perfectoid and $\overline{R}$ is $p$-torsion, both are of bounded $p^{\infty}$-torsion. Therefore by \Cref{CompBaseChange} (1) both base change are discrete. Now using \Cref{pfdideals_along_fflatextension} the kernel of the surjective maps 
\[R'\to T\widehat{\otimes}^{L}_R R' \quad R'\to \overline{R}\widehat{\otimes}^{L}_R R'\]
    are the $p$-adic closures of $R[p^{\infty}]R'$ and $\sqrt{pR}R'$. Note that $T\widehat{\otimes}^{L}_R R'$ is $p$-torsion free by \Cref{torsion_and_flatness}. Therefore, this implies that $R'[p^{\infty}]\subset \overline{R[p^{\infty}]R'}$. But note also that $R[p^{\infty}]R'\subset R'[p^{\infty}]$. As $R'[p^{\infty}]$ is $p$-adically closed (for example, the quotient by this ideal is perfectoid) we can conclude $\overline{R[p^{\infty}]R'}=R'[p^{\infty}]$ as desired. For $\overline{\sqrt{pR}R'}$ a similar argument holds: namely, as the quotient is perfect in characteristic $p$, we know that $\sqrt{pR'}\subset \overline{\sqrt{pR}R'}$. But also $\sqrt{pR}R'\subset \sqrt{pR'}$, which concludes as above.

\end{proof}

\subsection{Tameness of torsion in perfectoid rings and perfectoidizations}\label{tame_torsion_section}

We now prove \Cref{thm_torsion_intro}. This shows that torsion in perfectoid rings satisfies a strong tameness property.

    


\begin{thm}\label{torsion_is_almostzero}Let $R$ be a perfectoid ring and let \(I\) be an ideal of \(R\). Then the $I$-torsion in $R$ is $I_{\pfd}$-almost zero.
    
\end{thm}

\begin{proof}
Fix an element $x\in R$ such that $dx=0$ for any \(d \in I\). We will show that \(\alpha x = 0\) for any \(\alpha \in I_{\pfd}\).
By André's lemma \cite[Theorem 7.14]{bhatt2022Prismsa} we take $R\to R'$ a $p$-completely faithfully flat extension by a perfectoid ring with a compatible sequence $(d^{1/p^n})_{n\in \N}$ of \(p\)-power roots of all elements $d$ of \(I\) in $R'$.
By \Cref{Pfdizationideal_generated_by_p_power roots} the perfectoidization of $IR'$ is given by the $p$-adic closure of the ideal $\sum_{d \in I}\sum_{n\in \N}d^{1/p^n}R'$. 
Now because $R'$ is reduced we have that $d^{1/p^n}x=0$ for any $n\in \N$, see the discussion at the end of \cite[Section 2.1.3]{cesnavicius2023purityflatcohomology}. As a consequence, for any $\alpha\in (IR')_{\pfd}$ we have $\alpha x=0$. In particular, for any $\alpha\in I_{\pfd}\subset R$ we have $\alpha x=0$ in $R'$ since \(I_{\pfd}\) is contained in \((IR')_{\pfd}\) by \Cref{pfdideals_along_fflatextension}.
The injectivity of $R\to R'$, seen in  \Cref{fflat_classpcomplete} now concludes.
\end{proof}



As a corollary of \Cref{torsion_is_almostzero} we also prove the following which studies more in detail torsion phenomena in perfectoid rings. This is well known when the ideal $I=(d)$ is principal and $d=p$, see for example \cite[Section 2.1.2]{cesnavicius2023purityflatcohomology}.



    


\begin{proposition}\label{bounded_d_infinity_torsion}
Let $R$ be a perfectoid ring and $I\subset R$ an ideal. Then for all $J\subseteq I_{\pfd}$ such that $J_{\pfd}=I_{\pfd}$ then 
\[R[I]=R[J].\]
In particular, note that we can take $J=I^n$ for any $n\in \N$.
Moreover, $R[I]$ is a perfectoid ideal.
    
\end{proposition}
\begin{proof}Say $x\in R$ is such that $xd=0$ for all \(d \in I\). Then $x$ is $d$-torsion implying that for every $d'\in J \subseteq J_{\pfd} = I_{\pfd}$ we have $xd'=0$ by \Cref{torsion_is_almostzero}. Now the symmetry of the situation concludes for the claimed equality $R[I]=R[J].$

Now, note that if $(x_n)_{n\in \N}$ is a sequence of elements in $R[I]$ converging in $R$ for the $p$-adic topology to an element $x_{\infty}\in R$, such that $x_nd=0$ for all \(d \in I\), then $x_{\infty}d=0$ by continuity of the multiplication. This shows that $R[I]$ is a $p$-adically closed ideal. Also if $x\in R[I]$, then $xd=0$ for any \(d \in I\). So $d$ is $x$-torsion. As a consequence, $(x)_{\pfd}d=0$ by \Cref{torsion_is_almostzero}, implying that $(x)_{\pfd}\subset R[I]$. Now, it follows that the ideal is a perfectoid ideal using \Cref{pfd_ideals_sum} (2).
    
\end{proof}

Also analogous to the situation for the principal ideal $(p)$ (see \cite[Section 2.1.3]{cesnavicius2023purityflatcohomology}) we get the following corollary of \Cref{bounded_d_infinity_torsion}, using arc descent. This is \Cref{thm_decomposition_intro} in the introduction.

\begin{thm}\label{dtorsion_decomposition}Let $R$ be a perfectoid ring and $I\subseteq R$ an ideal. Then the following consequences hold.
\begin{enumerate}
    \item The map $R\to R/R[I]$ is a surjection to a perfectoid ring which is inital with respect to maps from $R$ to $I$-torsion-free perfectoid $R$-algebras. Moreover this map is a $I_{\pfd}$-almost isomorphism.
    \item We have a pullback square
    \[\begin{tikzcd}
R \arrow[d] \arrow[r]  & R/R[I] \arrow[d]                 \\
R/I_{\pfd} \arrow[r] & {R/(R[I],I_{\pfd})}
\end{tikzcd}\]
In other words we have a canonical decomposition of the ring $R$
\[R\to R/R[I]\times_{R/(R[I],I_{\pfd})}R/I_{\pfd}\]
into a $I$-torsion-free perfectoid ring and a $I$-torsion perfectoid ring.
\end{enumerate}
    
\end{thm}
\begin{proof}
    The first point follows entirely from \Cref{bounded_d_infinity_torsion}. As for the second point it is a direct application of \Cref{812prisms} for the map $R\to R/R[I]$ which is a $I_{\pfd}$-almost isomorphism, in particular an isomorphism away from \(I\), using that $(R[I],I_{\pfd})$ is a perfectoid ideal by \Cref{pfd_ideals_sum} (3).
\end{proof}

\subsection{Application to existence of perfectoidizations}\label{exist_pfd_sec}

We use \Cref{dtorsion_decomposition} and the arc topology to deduce the following theorem about existence of perfectoidization.

\begin{proposition}\label{I_tors and I_tors_free existence} Let $S$ be a ring and $I\subset S$ an ideal. Then the following are equivalent.

\begin{enumerate}
    \item There is a perfectoid $S$-algebra $S\to S_{\pfd}$ which is initial with respect to maps from $S$ to perfectoid rings.
    \item There are perfectoid $S$-algebras $S\to S_{I\tf}$ and $S\to S_{I\tors}$ such that $S\to S_{I\tf}$ is initial with respect to maps from $S$ to $I$-torsion-free perfectoid algebras, and $S\to S_{I\tors}$ is initial with respect to maps from $S$ to $I$-torsion perfectoid algebras.
\end{enumerate}
    
\end{proposition}

\begin{proof} If (1), holds, then \Cref{dtorsion_decomposition}, concludes using respectively
\begin{equation*}
    S_{\pfd}/S_{\pfd}[IS_{\pfd}] \quad \text{and} \quad S_{\pfd}/(IS_{\pfd})_{\pfd}.
\end{equation*}

For the converse, note that
\[\begin{tikzcd}
S \arrow[d] \arrow[r] & {S/S[I^{\infty}]} \arrow[d] \\
S/IS \arrow[r]        & {S/(I,S[I^{\infty}])}     
\end{tikzcd}\]
becomes a pullback square as in \Cref{812prisms} because $S\to S/I \times S/S[I^{\infty}]$ is an arc-cover. But note that, using the end of \Cref{pfdization_def} and our hypothesis $(S/I)_{\pfd}=S_{I\tors}$ and $(S/S[I^{\infty}])_{\pfd}=S_{I\tf}$. Also, using the universal property, we see that $(S/(I,S[I^{\infty}]))_{\pfd}=S_{I\tf}/(IS_{\tf})_{\pfd}$ since the right hand side being the universal $I$-torsion perfectoid $S_{I\tf}$-algebra. Therefore we have a pullback square of $E_{\infty}$-rings
\[\begin{tikzcd}
S_{\pfd} \arrow[d] \arrow[r] & S_{I\tf} \arrow[d]         \\
S_{I\tors} \arrow[r]         & S_{I\tf}/(IS_{\tf})_{\pfd}
\end{tikzcd}\]
But as the three terms except the top left are discrete in this pullback, we conclude that the same goes for $S_{\pfd}$. Now, using \cite[Proposition 2.1.4]{cesnavicius2023purityflatcohomology} that the pullback $S_{\pfd}$ is perfectoid. By the end of \Cref{pfdization_def}, we deduce that $S_{\pfd}$ is an initial perfectoid $S$-algebra.
\end{proof}

\begin{rem}\label{I=p existence of pfd} In \Cref{I_tors and I_tors_free existence}, if $I\supset pS$ then the existence of $S_{I\tors}$ is free: indeed, it is perfection of the characteristic $p$-algebra $(S/I)_{\perf}$. As a corollary, for a ring $S$, there is an initial perfectoid $S$-algebra if and only there is a $p$-torsion-free perfectoid $S$-algebra.
    
\end{rem}

As an application of \Cref{I=p existence of pfd}, we prove the following decomposition of the perfectoidization.
This originally belongs to the present work. However, as it was needed in \cite[Proposition 2.7]{ishizuka2025Graded}, a paper co-authored by the first-named author, it was first stated there without proof, announcing that its proof would be provided in a future joint work by the present authors. We now give the complete proof.

\begin{cor} 
\label{DealingWithTorsion}
Let $R$ be a ring, and suppose that there exist an initial map to a perfectoid ring $R\to R_{\pfd}$. Denote $R'=R/R[p^{\infty}]$. 
Then \((R')_{\pfd}\) exists and the natural map to a fiber product of discrete rings
\[R_{\pfd}\to (R')_{\pfd}\times_{\overline{(R')_{\pfd}}}(R/pR)_{\perf}\]
is an isomorphism, where $\overline{(-)}$ denotes the reduced modulo $p$-fiber.
\end{cor}
\begin{proof}
    The existence of \((R')_{\pfd}\) as a discrete ring comes from \Cref{I_tors and I_tors_free existence} since \(R' \to R_{\pfd}/R_{\pfd}[p^\infty]\) is initial among maps to \(p\)-torsion-free perfectoid \(R'\)-algebras.
    Moreover, \((R')_{\pfd}\) is a \(p\)-torsion-free perfectoid ring by \cite[Lemma A.2]{ma2022Analogue} (see also \Cref{pfd_is_p_torsionfree}) and the map \(R \to (R')_{\pfd}\) is initial among \(p\)-torsion-free perfectoid \(R\)-algebras.
    Therefore, \Cref{I_tors and I_tors_free existence} says that \(R_{\pfd}\) can be identified with the map \((R')_{\pfd}\times_{\overline{(R')_{\pfd}}}(R/pR)_{\perf}\).
    

    
\end{proof}

As an application of \Cref{DealingWithTorsion}, we compute the perfectoidization of semiperfectoids of the form $R/(pr)$ when $R/(r)$ and $R$ are $p$-torsion-free.
\begin{proposition}\label{ptorsion_semipfd_typical}Let $R$ be a perfectoid ring and $r\in R$ such that $R/(r)$ is a $p$-torsion-free ring. Then the natural map
\[R/(pr)\to R/(r)_{\pfd}\times_{\overline{R}/(\overline{r}^{1/p^{\infty}})}\overline{R}\]
is the perfectoidization map. Here $\overline{R}$ denotes the perfection (reduction) of $R/pR$ and $(\overline{r}^{1/p^{\infty}})$ denotes the radical of the ideal generated by the image of $r$ in $\overline{R}$.
\end{proposition}
\begin{proof}
    Using \Cref{DealingWithTorsion}, we need to show that the $p$-torsion-free perfectoidization is $R/(r)_{\pfd}$. First, note that $r$ is of $p$-torsion in $R/(pr)$. So $r$ is zero in the $p$-torsion-free perfectoidization. Therefore, note that if $R/(pr)\to S$ is a map to $p$-torsion-free perfectoid ring $S$, then has $r$ is sent to zero we see that it factors uniquely through $R/(r)_{\pfd}$. Note that has $R/(r)$ is $p$-torsion-free by assumption, then $R/(r)_{\pfd}$ is also $p$-torsion-free by \cite[Lemma A.2]{ma2022Analogue} -- see also \Cref{pfd_is_p_torsionfree}. This observation concludes.
\end{proof}

\begin{example}
\label{TorsionExample}

We will use \Cref{DealingWithTorsion} to compute the perfectoidization of $R\abracket{x^{1/p^{\infty}}}/(px)$.


We want to compute the initial $p$-torsion-free perfectoid ring $R'$ with a map $R\abracket{x^{1/p^{\infty}}}/(px)\to R'$.
As the image of $x$ is $p$-torsion in $R\abracket{x^{1/p^{\infty}}}/(px)$ and that every perfectoid ring is reduced \cite[Subsection 2.1.3]{cesnavicius2023purityflatcohomology}, we see that the quotient map $R\abracket{x^{1/p^{\infty}}}/(px)\to R$ evaluating every $x^{1/p^n}$ at zero for $n\geq 0$ has to factorize the $p$-torsion-free perfectoidization $R'$. But as this quotient is already perfectoid and $p$-torsion-free by assumption, this shows that this map \(R\abracket{x^{1/p^\infty}}/(px) \to R\) is the initial map to $p$-torsion-free perfectoid rings.

Using \Cref{DealingWithTorsion}, we can conclude that the perfectoidization of the semiperfectoid ring $R\abracket{x^{1/p^{\infty}}}/(px)$ is identified to the map
\[R\abracket{x^{1/p^{\infty}}}/(px)\to R\times_{R/\sqrt{p R}}R/\sqrt{p R}[x^{1/p^{\infty}}],\]
that sends the class in the quotient of a power series $f\in R\abracket{x^{1/p^{\infty}}}$ to $(f(0),\overline{f})$ where $\overline{f}$ denotes the reduced modulo $p$ quotient.

Therefore we deduce that (see \Cref{intersection_product_pfdization})
\[(px)_{\pfd}=(p)_{\pfd}\cap(x)_{\pfd},\]
the kernel of the natural map from $R\abracket{x^{1/p^{\infty}}}$ to the product $R\times R/\sqrt{p R}[x^{1/p^{\infty}}]$.
In words, this ideal $(px)_{\pfd}$ consists of power series in $R\abracket{x^{1/p^{\infty}}}$ such that every coefficient of the power series is in $\sqrt{p R}$ and the constant coefficient is equal to zero.
\end{example}

As an application we also get a general formula for the perfectoidization of a semi-perfectoid ring, using \Cref{semipfdproot}. The second-named author uses this result in another paper (\cite[Corollary 4.12]{ishizuka2025Graded}) in which the proof crucially relies on \Cref{DealingWithTorsion}. So for completeness, we give the proof as follows.

\begin{proposition}\label{semipfd_gen_computation}
Let \(S\) be a ring whose \(p\)-adic completion \(\widehat{S}\) is a semiperfectoid ring.
Let $\widetilde{S}$ be the $p$-completion of the $p$-root closure of $S$ in $S[\frac{1}{p}]$. Then the perfectoidization identifies to
\[S\to \widetilde{S}\times_{(\widetilde{S}/p\widetilde{S})_{\red}}(S/pS)_{\red}.\]  
\end{proposition}
\begin{proof}
    Now $S'=S/S[p^{\infty}]$ is a $p$-torsion-free and its \(p\)-adic completion \(\widehat{S'}\) is a quotient of a semiperfectoid ring, therefore \(\widehat{S'}\) is a semiperfectoid ring. 
    Since the image of \(S\) in \(S[1/p]\) is equal to \(S'\), the taken \(\widetilde{S}\) is the \(p\)-adic completion of the \(p\)-root closure of \(S'\).
    Therefore, the perfectoidization of $S'$ is $\widetilde{S}$ by \Cref{semipfdproot}. But as any map from $S$ to a $p$-torsion-free ring uniquely factorizes through $S\to S'$, we conclude by \Cref{DealingWithTorsion}.
\end{proof}

\end{document}